\documentclass[]{article}
\usepackage[affil-it]{authblk}

\usepackage{fullpage}
\usepackage{graphicx}
\usepackage{pdflscape}
\usepackage{secdot}
\usepackage{subcaption}
\usepackage{amsmath, amssymb,amsthm}
\usepackage{amsfonts}
\usepackage{latexsym}
\usepackage{booktabs}
\usepackage{enumitem,blindtext}
\usepackage[authoryear]{natbib}
\usepackage[export]{adjustbox}
\usepackage{mathrsfs}
\usepackage[usenames,dvipsnames]{color}
\usepackage[mathcal]{eucal}
\usepackage{caption}
\usepackage{amsmath,amsfonts}
\usepackage[linesnumbered, noend, ruled]{algorithm2e}
\usepackage{setspace}
\newcommand{\comments}[1]{}
\DeclareGraphicsExtensions{.png,.pdf,.jpg}

\newtheorem{observation}{Obervation}

\usepackage{tikz}
\usepackage{tikzsymbols}
\usetikzlibrary{arrows}
\usetikzlibrary{calc}

\newcommand{\WFS}{WFS\xspace}
\newcommand{\fT}{\mathcal T}

\renewcommand{\paragraph}[1]{{\vskip2mm\noindent \bf #1}\hspace{0.4cm}}

\definecolor{purple}{RGB}{153,50,204}

\usepackage[normalem]{ulem}

\usepackage{xstring}

\usepackage{hyperref}

\newcommand*\GetListMember[2]{\StrBetween[#2,\number\numexpr#2+1]{,#1,},,\par}%

\newlength{\MidRadius}
\newcommand*{\CircularSequence}[3]{%
	\StrCount{#3}{,}[\NumberOfElements]
	\pgfmathsetmacro{\AngleSep}{360/(\NumberOfElements+1)}
	\pgfmathsetlength{\MidRadius}{(#1+#2)/2}
	\draw [black,  ultra thick] circle (#2);
	\draw [black, ultra thick] circle (#1);
	\foreach [count = \Count] \Angle in {0,\AngleSep,..., 360} {%
		\draw [black, ultra thick] (\Angle:#2) -- (\Angle:#1);
		\pgfmathsetmacro{\MidPoint}{\Angle+\AngleSep/2}
		\node at (\MidPoint:\MidRadius) {\GetListMember{#3}{\Count}};
	}%
}%

\begin{document}

\title{An iterative exact algorithm for the weighted fair sequences problem}

\author[1]{Markus Sinnl\thanks{markus.sinnl@jku.at}}

\affil[1]{Institute of Production and Logistics Management/JKU Business School, Johannes Kepler University Linz, Linz, Austria}

\date{}

\maketitle

\begin{abstract}
In this work, we present a new iterative exact solution algorithm for the weighted fair sequences problem, which is a recently introduced NP-hard sequencing problem with applications in diverse areas such as TV advertisement scheduling, periodic machine maintenance and real-time scheduling. In the problem we are given an upper bound on the allowed solution sequence length and a list of symbols. For each symbols, there is a positive weight and a number, which gives the minimum times the symbol has to occur in a feasible solution sequence. The goal is to find a feasible sequence, which minimizes the maximum weight-distance product, which is calculated for each consecutive appearance of each symbol in the sequence, including the last and first appearance in the sequence, i.e., the sequence is considered to be circular for the calculation of the objective function.

Our proposed solution algorithm is based on a new mixed-integer programming model for the problem for a fixed sequence length. The model is enhanced with valid inequalities and variable fixings. We also develop an extended model, which allows the definition of an additional set of valid inequalities and present additional results which can allow us to skip the solution of the mixed-integer program for some sequence lengths.

We conduct a computational study on the instances from literature to assess the efficiency of our newly proposed solution approach. Our approach manages to solve 404 of 440 instances to optimality within the given timelimit, most of them within five minutes. The previous best existing solution approach for the problem only managed to solve 229 of these instances, and its exactness depends on an unproven conjecture. Moreover, our approach is up to two magnitudes faster compared to this best existing solution approach.
\end{abstract}

\section{Introduction and motivation}

Scheduling and sequencing problem are amongst the most important and well-studied problems in Operations Research. In such problems, we are typically given a set of tasks, which use up resources, and one or more machines, and we have to allocate the tasks to the machine(s), while fulfilling constraints on the allowed allocations (e.g., starting times, conflicts or dependencies between tasks, \ldots) and considering some objective function (e.g., minimize the finishing time of the last task, \ldots). For a general introduction to scheduling and sequencing, we refer to the book \citep{Brucker2004} and the recent surveys \citep{allahverdi2008survey,levner2010complexity,jimenez2013survey}. In this paper, we study the recently introduced \emph{weighted fair sequences problem (\WFS)} \citep{pessoa2018weighted}, which is defined as follows. 

%\begin{definition}
Let $A=\{a_1,a_2,\ldots a_n\}$ be a set of $n$ symbols with a weight function $w:A\rightarrow \mathbb Z^+$. Let $T$ be the maximum length of the solution sequence and let $f_i$, $a_i \in A$ be minimum number of times symbol $a_i$ must occur in the solution sequence. A feasible solution to the \WFS consists of a sequence $S=s_1,s_2,\ldots, s_{T'}$ with $T'\leq T$ and $s_i \in A$, such that each symbol $a_i$ occurs at least $f_i$ times in $S$. The objective function value is based on a circular distance function. Let $a_i(S)=(p^1_{a_i},p^2_{a_i},\ldots,p^k_{a_i})$ be the ordered list of positions in sequence $S$, where symbol $a_i$ occurs. For each consecutive pair of entries $p^\ell_{a_i},p^{\ell+1}_{a_i} $, the distance $dist(p^\ell_{a_i},p^{\ell+1}_{a_i})=p^{\ell+1}_{a_i}-p^\ell_{a_i}$ is calculated, moreover, the distance $dist(p^k_{a_i},p^{1}_{a_i})=T'+p^{1}_{a_i}-p^k_{a_i}$ is also calculated. Note that this last distance-calculation can also be seen as calculation between consecutive pairs in $a_i(S)$ by considering the sequence $S$ to be circular. Let $D_i$ be the maximum of these distances for $a_i \in A$. The goal is to find a sequence $S^*$ which minimizes $w(a_i)D_i$, $a_i \in A$ (i.e., the weight-distance-product).  

%be the number of times symbol $a_i$ occurs in $S$. The distance function $dist$ is based on the distance between the $k$-th and $k+1$-st position of each symbol in $S$, where $k$, $k+1$ is considered modulo $|a_i(S)|$. Let $pos(a_i,k,S)$ and $pos(a_i,k+1,S)$ denote these positions. The distance is defined as 

%\begin{equation*}
%dist(a_i,k,k+1,S)=\begin{cases}
%pos(a_i,k+1,S)-pos(a_i,k,S) &  if k+1\ mod\ |a_i(S)| > k\ mod\ |a_i(S)| \\
%T'+pos(a_i,k,S)-pos(a_i,k+1,S) & otherwise.
%\end{cases}
%\end{equation*}
	
%\end{definition}

Figure \ref{fig:instance} shows two solutions of an instance of the \WFS. The solution sequences are displayed in a circular fashion to give an easy illustration of the circular objective function. The figure shows how placing a symbol $a_i$ in a sequence in a larger number than required by $f_i$ can improve the objective function value.
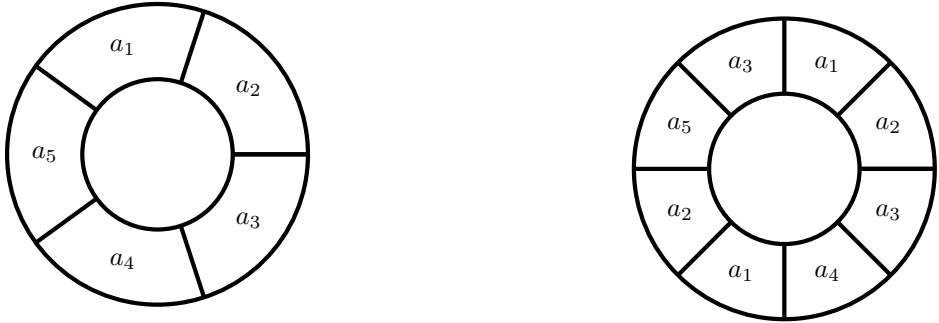
\begin{figure}
\begin{subfigure}{.5\linewidth}
	\centering
\begin{tikzpicture}
\CircularSequence{2.0cm}{1.cm}{$a_2$,$a_1$,$a_5$,$a_4$,$a_3$}
\end{tikzpicture}
\caption{Solution with objective value $50$ (due to $w(a_1)D_1$ and $w(a_2)D_2$)}
\end{subfigure}
\begin{subfigure}{.5\linewidth}
	\centering
	\begin{tikzpicture}
	\CircularSequence{2.0cm}{1.cm}{$a_2$,$a_1$,$a_3$,$a_5$,$a_2$,$a_1$,$a_4$,$a_3$}
	\end{tikzpicture}
	\caption{Optimal solution, objective value $48$ (due to $w(a_4)D_4$)}
\end{subfigure}

\caption{Two solutions for an instance of the \WFS with $A=\{a_1,a_2,a_3,a_4,a_5\}$, $T=10$, $f_i=1, a_i \in A$ and $w(a_1)=10$, $w(a_2)=10$, $w(a_3)=7$, $w(a_4)=6$, $w(a_5)=3$.  \label{fig:instance}}

\end{figure}

The \WFS is NP-hard \citep{pessoa2018weighted} and falls within the class of scheduling problems, which are concerned with the proportional sharing of resources over time. In these so-called \emph{fair sequencing problems} \citep{kubiak2004fair}, the same tasks are executed successively, and the goal is to obtain a schedule fulfilling certain \emph{fairness criteria}, e.g., the (temporal) distance of successive executions of the same tasks should be similar for all type of considered tasks. They play an important role in Just-In-Time manufacturing systems such as the ones used by Toyota, who were among the first to realize the potential of fair sequences in a production context, as it helped to "smooth out" the usage rate of all parts of their assembly line \citep{kubiak2004fair,monden2011toyota}. Moreover, the concept of \emph{fair sequences} also has interesting applications in the area of political representation \citep{tijdeman1980chairman} and resource allocation in computer operating systems \citep{waldspurger1995stride}. Additional applications of the \WFS discussed in \citep{pessoa2018weighted} include TV advertising \citep{bar2003windows,garcia2015scheduling} and periodic machine maintenance \citep{anily1998scheduling,bar2002minimizing}.  

%We note that there are many different concepts and versions of \emph{fairness} and also many different objective functions considered in these kind of problems. 

%We refer to the paper introducing the \WFS \citep{pessoa2018weighted} for an extensive literature overview of problems related to fair sequencing.

\subsection{Contribution and outline}

In this paper, we introduce a new iterative exact algorithm for the \WFS. The algorithm is based on a new mixed-integer programming (MIP) model for the \WFS\ for sequences which have exactly length $\fT$.
In Section \ref{sec:algorithm} we describe our solution algorithm, including the MIP model, and also present valid inequalities and variable fixing procedures for the MIP model. An extended MIP model using additional variables is also developed and additional algorithmic enhancements are discussed. In Section \ref{sec:results} we present a computational study using the instances from \cite{pessoa2018weighted}. The results show that our algorithm outperforms the algorithm from \cite{pessoa2018weighted} by up two orders of magnitude. We are able to solve 404 of 440 instances from literature to proven optimality, most of them within fives minutes. Section \ref{sec:conclusion} concludes the paper. In the remainder of this section, we give a short overview of the model and algorithms for the \WFS\ proposed in \cite{pessoa2018weighted}. 

\subsection{The solution approach of \cite{pessoa2018weighted}}

In \cite{pessoa2018weighted} a MIP model for the problem is introduced. As the model has a very large number of variables and constraints the authors also propose an alternative solution approach, in which a reduced version of their MIP model is solved iteratively. 

Their MIP model has nine different sets of variables and can be viewed as a combination of a model calculating the distances within the same cycle, and a model calculating the distances between different cycles. The first part consists of nine different set of constraints, and the second part consists of seven different sets of constraints, not counting binary constraints on variables. To reduce the size of the model, the authors introduce a nonlinear MIP model which can provide an upper bound on the number of copies for each symbol in the optimal solution. This upper bound can be used to reduce the number of variables in their model. However, their MIP model for the \WFS after applying the variable-reduction is still too large to be effective in computations, even after adding some valid inequalities proposed by the authors. Thus they also develop an iterative solution approach, which is based on an estimated upper bound for the optimal solution value. With this estimated upper bound, the number of variables can be further reduced in combination with their nonlinear MIP model. The iterative algorithm starts with an initial upper bound, and then tries to solve their MIP model with this estimated upper bound set as cutoff for the objective function value. If the problem with the cutoff is feasible, the optimal solution will be the optimal solution for the \WFS problem. If not, the estimated upper bound is increased, and the procedure repeated. This is done iteratively until either the optimal solution is found at some point, or a given iteration limit is reached. We note that in the main computations in \cite{pessoa2018weighted}, the used MIP model contains a set of inequalities for which the authors only conjectured validity, but did not prove it. Thus the reported results of \cite{pessoa2018weighted} can only be considered heuristic in nature until validity of this inequality is proven.

\section{Solution algorithm  \label{sec:algorithm}}

Our solution algorithm is described in Algorithm \ref{alg}. 
\begin{algorithm}[h!tb]   
	\DontPrintSemicolon                 
	\caption{Iterative solution algorithm for the \WFS. \label{alg}}
	\label{alg:sepmaxviolated}
	$z^*\gets \infty$\;
	\For{$\sum_{a_i \in A} f_i \leq \fT\leq T $}
	{
		$z' \gets $ optimal value for \WFS with fixed sequence length $\fT$\; \label{alg:line3}
		\If{$z'<z^*$}
		{
			$z^* \gets z'$\;
		}
	}

\end{algorithm}

The algorithm iteratively solves the \WFS for a fixed sequence length $\fT$, and iterates over all possible $\fT$, i.e., $\sum_{a_i \in A} f_i \leq \fT \leq T$. To do so, the MIP model described in Section \ref{sec:model} is used. The model can be strengthened throughout the algorithm by information from previous iterations, this is discussed, among other possible strengthenings of the model, in Section \ref{sec:strength}. Moreover, we also present some results in Section \ref{sec:strength} which allow us to skip solving the WFS for some $\fT $.

By considering a fixed $\fT$ the modeling of the problem can be simplified, allowing for a model with a much smaller number of variables and constraints compared to the model used in \cite{pessoa2018weighted} (including their reduced variant they use in their iterative algorithm). We note that there also exist other problems where fixing a certain dimension of the problem and then reformulating the problem and solving the reformulated problem in an iterative fashion leads to a more effective solution approach compared to directly solving the problem, see, e.g., \cite{chen2009new,contardo2019scalable,sinnl2020note}.

\subsection{MIP model for a fixed sequence length $\fT$ \label{sec:model}}

Let binary variables $x_{i}^t$, $a_i \in A, t=1,\ldots, \fT$ indicate if symbol $a_i$ gets placed at position $t$ in the sequence. Let binary variables $s_i^{t,t'}$, $a_i \in A$, $1\leq t, t'\leq \fT$ indicate that if symbol $a_i$ gets placed at position $t$, the next occurrence of symbol $a_i$ is at position $t'$, considering a circular sequence. Let binary variables $p_i^{t,t'}$, $a_i \in A$, $1\leq t, t'\leq \fT$ indicate that if symbol $i$ gets placed at position $t$, the last occurrence of symbol $i$ before $t$ is at position $t'$, considering a circular sequence. Note that these variables actually encode a relaxed version of the described indications in our model, i.e., for some combinations of $i,t,t'$, $s$-variables or $p$-variables which do not correspond to the next/last occurrence given the solution encoded by $x$ can be selected. This does not matter, as these variables are used for measuring the objective function value, and our objective function has a min/max-structure, thus only certain $i,t,t'$ need to be selected correctly (i.e., the ones corresponding to symbols $a_i$ and positions $t,t'$ causing the maximum weight-distance-product for the sequence encoded by the $x$-variables). More details on this is given below after the MIP model. Let integer variable $\theta$ measure the objective function value. Let
\begin{equation*}
dist(t,t')=\begin{cases}
t-t' &  if\ t>t' \\
\fT+t-t' & otherwise.
\end{cases}
\end{equation*}

The \WFS\ for a fixed sequence length $\fT$ can be formulated as follows.

\begin{align}
\min \quad &\theta& \\
\sum_{1 \leq t\leq \fT } x^t_i &\geq f_i & \quad \forall a_i \in A \label{eq:chosenmin}\\
%\sum_{t \in 1,\ldots, \fT} x^t_i &\leq M & \quad \forall i \in X \label{eq:chosenmax}\\
%\sum_{t \in 1,\ldots, \fT} d_i^t &=1 & \quad \forall i \in X \label{eq:fixdistance}\\
\sum_{a_i \in A} x^t_i &=1  &\quad 1 \leq t\leq \fT  \label{eq:position} \\
x^t_i-\sum_{1 \leq t'\leq \fT } p_i^{t,t'}&=0  &\quad \forall a_i \in A,  1 \leq t\leq \fT  \label{eq:pred} \\
x^t_i-\sum_{1 \leq t'\leq \fT } s_i^{t,t'}&=0  &\quad \forall a_i \in A, 1 \leq t\leq \fT  \label{eq:succ} \\
s_i^{t,t'}-p_i^{t',t}&=0 &\quad \forall i \in A_i, 1 \leq t,t'\leq \fT  \label{eq:predsuc} \\
\theta - \sum_{a_i \in A} \sum_{1\leq t'\leq \fT} w(a_i) dist(t,t') p^{t,t'}_i &\geq 0 & 1\leq t \leq \fT \label{eq:objpred} \\
\theta - \sum_{a_i \in A} \sum_{1\leq t'\leq \fT} w(a_i) dist(t',t) s^{t,t'}_i &\geq 0 & 1\leq t \leq \fT \label{eq:objsucc} \\
x^t_i&\in \{0,1\}&\quad \forall a_i \in A, 1 \leq t\leq \fT \notag \\
p^{t,t'}_i&\in \{0,1\}&\quad \forall a_i \in A, 1 \leq t,t'\leq \fT \notag \\
s^{t,t'}_i&\in \{0,1\}&\quad \forall a_i \in A, 1 \leq t,t'\leq \fT \notag 
\end{align}

Constraints \eqref{eq:chosenmin} ensure that each symbol $a_i \in A$ is appearing in the sequence at least $f_i$-times. Constraints \eqref{eq:position} ensure that each position in the sequence only get assigned a single symbol.  Constraints \eqref{eq:pred} and \eqref{eq:succ} ensure that for each placed symbol in the sequence at a certain position, an indicator-variables for the next and previous occurrence of the symbol is selected. Constraints \eqref{eq:predsuc} link the $s$-variables with the $p$-variables. Note that constraints  \eqref{eq:pred}, \eqref{eq:succ}, \eqref{eq:predsuc} only enforce that for each symbol $a_i$ placed at a position $t$, some position $t'$, where symbol $a_i$ is also placed is defined as "predecessor" and "successor". However, on their own, they do not enforce that these selected positions are the truly the previous/next occurrence of the symbol in the solution sequence. In fact, our model does not enforce that this must hold for all the $s$-variables and $p$-variables, but only for the symbol-position-combinations causing the largest weight-distance-product value for the sequence encoded by the $x$-variables. As the objective in the \WFS is concerned with minimizing the maximum weight-distance-product, this is enough to solve the problem to proven optimality. The two set of constraints \eqref{eq:objpred} and \eqref{eq:objsucc}, together with the minimization objective function ensure that this value is measured correctly. For each position $1\leq t \leq \fT$, in the sequence, constraint \eqref{eq:objpred} measure the weight-distance-product induced by the symbol $a_i$ placed at this position $t$, and the position $t'$ encoded as predecessor position by $p^{t,t'}_i$, and ensure that $\theta$ takes at least this value. Now suppose that there is a solution $(x^*,s^*,p^*,\theta^*)$, where $p^{*t,t'}_i$ is not selected in such a way that $t'$ is the predecessor position for the symbol placed at position $t$ for a constraint \eqref{eq:objpred} fulfilled with equality by $(x^*,s^*,p^*,\theta^*)$. Clearly, another feasible solution  $(x^*,s',p',\theta')$ also exists, where $p'^{t,t'}_i$ is selected in such a way that $t'$ is the predecessor position for the symbol placed at position $t$ for this constraint, and $\theta'\leq \theta^*$, as the selection of $p'^{t,t'}_i$ in such a way makes the sum in \eqref{eq:objpred} as small as possible. As the objective function consists of minimizing $\theta$, this ensures that for all constraints \eqref{eq:objpred} fulfilled with equality in a solution the $p$-variables are selected correctly, and thus the objective is measured correctly. Constraints \eqref{eq:objsucc} do the same for the $s$-variables and the successors in the sequence.
 
%  These constraints measure the objective function value based on the $s$-variables and the $p$-variables. As the objective function is minimization, and any selection of a $s_i^{t,t'}$, where $t'$ is not the next occurrence of the symbol $a_i$ in the (circular) sequence would lead to a larger objective function value compared to 
 
\subsection{Valid inequalities, variable fixings and other enhancements for the solution algorithm\label{sec:strength}}

The following set of inequalities is valid for our MIP model in the sense that there exists at least one optimal solution fulfilling it.

\begin{observation}
Let $M_i$ be an upper bound one the number of times symbol $a_i$ can occur in an optimal solution. Then the inequality
\begin{equation}
\sum_{1\leq t \leq \fT} x^t_i \leq M_i \label{eq:ub}
\end{equation}
is valid for all $a_i \in A$.
\end{observation}

As the upper bound $M_i$ the value $min(\fT/2,\fT-n-1)$ can be used, which is an upper bound for which at least one optimal solution to the \WFS exists, see \cite{pessoa2018weighted}. This upper bound is based on the observation that there exists an optimal solution, in which no symbol $a_i$ will be placed at both position $t$ and $t+1$ (including the $\{\fT,1\}$ combination of positions). 
%In our algorithm, we also consider $M_i=\lfloor z^*/w(a_i) \rfloor$, where is a given an upper bound $z^*$ for the optimal solution value of an \WFS instance (e.g., the best solution value encountered in the previous iterations of our algorithm). 

This observation of \cite{pessoa2018weighted} can also be used to fix some of the $s$-variables and $p$-variables. Correctness of this fixing follows from the fact that a value of one for these variables would mean that a symbol $a_i$ is placed at two consecutive positions in the sequence.

\begin{observation}\label{obs:nearfixing}
All	variables $p^{t,t'}_i, s^{t,t'}_i$ with $|t-t'|=1$ can be fixed to zero. Moreover, variables $p^{1,\fT}_i, s^{1,\fT}_i$ and $p^{\fT,1}_i, s^{\fT,1}_i$ can also be fixed to zero.
\end{observation}

The next observation describes some variable fixings which can be done in our MIP model given an upper bound $z^*$ for the optimal solution value of an \WFS instance
\begin{observation}\label{obs:objfixing}
	Let $z^*$ be a given upper bound on the optimal solution value of a \WFS instance. Then all variables $p^{t,t'}_i$ with $w(a_i)dist(t,t')\geq z^*$ and all variables $s^{t,t'}_i$ with $w(a_i)dist(t',t) \geq z^*$ can be fixed to zero.
\end{observation}

Note that we assume that we also have a feasible solution $S(z^*)$ associated with this upper bound, e.g., in the context of an iteration of our algorithm, this is the best solution found in the previous iterations. As initial solution before the first iteration, any heuristic solution can be used. In our approach, we use the trivial solution, which just puts $f_1$ copies of symbol $a_1$ at the beginning of the sequence, followed by $f_2$ copies of symbol $a_2$ and so on.

With a feasible solution available, we can pose constraints such as the one in Observation \ref{obs:objfixing} that in the subsequent iterations only solutions which improve the objective value compared to $z^*$ are feasible.
This avoids the search for an optimal solution for the MIP in an iteration with a $\fT$, where no improving solution can be found. In fact, we also implicitly add a constraint such that only solutions where the solution value $\theta$ improves on $z^*$ are feasible. To do so, we calculate $\theta(z^*)=\max \{\theta': k_i\geq f_i, w(a_i)k_i<z^*,w(a_i)k_i<\theta', \forall a_i \in A; k \in \mathbb N^n \} $, i.e., we look for the maximum positive integer multiple of the weights which is smaller than the current $z^*$ and add $\theta\leq\theta(z^*)$ to our MIP. When these constraints only allowing improved solutions are added to the MIP, in line \ref{alg:line3} of our algorithm, we will either get that the MIP is infeasible, or if it is feasible, we are sure to have found an improved solution.

Given an upper bound $z^*$ for the optimal solution value of an \WFS instance, we can also potentially lift constraints \eqref{eq:chosenmin} by increasing the right-hand-side of them.
\begin{observation}\label{obs:liftingrhs}
 Let $k^*_i=\min\{k_i: w(a_i)\lceil\fT/k_i \rceil <z^*;k_i\geq f_i; k_i \in \mathbb N\}$. As we want to find an improving solution, we can set $f_i=k_i$ on the right-hand-side of constraints \eqref{eq:chosenmin}.
\end{observation}
\begin{proof}
 Suppose in a solution we would place symbol $a_i$ $k_i$-times in a sequence $S$, i.e., $|a_i(S)|=k_i$. Then the best possible distance (i.e., the minimal possible maximal distance between all its consecutive occurrences) for this symbol in any solution sequence is $D'_i=\lceil\fT/k_i \rceil$. This value is achieved by evenly distributing the placement of the symbol in the sequence. Due to the circular distance function, any non-even distribution would lead to a minimal maximal distance as least as large as can be seen by an exchange argument. 
\end{proof}

The next observation is based on Observation \ref{obs:liftingrhs} and can allow us to conclude for a given $\fT$ that there can be no improving solution without the need of solving our MIP.

\begin{observation}\label{obs:skipping}
	Let $K^*=\sum_{a_i \in A} k^*_i$. If $K^*> \fT$ then there exists no improving solution with sequence length $\fT$.
\end{observation}

Observation \ref{obs:liftingrhs} also allows us to lift the right-hand-side of valid inequalities \ref{eq:ub}.

\begin{observation}\label{obs:ub2}
	For a given $a_i \in A$, let $M^*_i=\fT-K^*+k^*_i$. Then an improving solution fulfills
\begin{equation}
\sum_{1\leq t \leq \fT} x^t_i \leq M^*_i \label{eq:ub2}
\end{equation}

\begin{proof}
	Suppose there are more than $M^*_i$ copies of symbol $a_i$ in a sequence. Then there are only less than $\fT-M^*_i$ positions available for the other symbols, but we need $\sum_{a'_i \in A: a'_i \neq a_i} k^*_i\geq \fT-M^*_i$ for the existence of an improving solution.
\end{proof}
	
\end{observation}

 %Let $k^*_i=\min\{k_i: w(a_i)\lceil\fT/k_i \rceil <z^*;k_i\geq f_i; k_i \in \mathbb N\}$. As we want to find an improving solution, we can set $f_i=k_i$ on the right-hand-side of constraints \eqref{eq:chosenmin}.

Due to the structure of the \WFS, there can be many symmetric solutions, i.e., if there are two or more symbols with the same $f_i$ and $w(a_i)$-values, they can be exchanged in a solution without changing the solution value. Moreover, due to the circular distance function, any solution which is obtained by shifting the positions of all the symbols in the solution has the same objective function value. To break these symmetries, we propose to use the following two sets of constraints. For posing these constraints, without loss of generality, we assume that the symbols are ordered in descending order by $w$, and if there are multiple symbols with the same way, these symbols are ordered in descending order by $f$. First, we fix that the symbol $a_1$ must be placed in the first position of the sequence, i.e.,
\begin{equation}
x^1_1=1, \label{eq:sym}
\end{equation}
to deal with the rotational symmetry of our solutions. Moreover, we impose an ordering on the symbols with the same $w(a_i)$ and $f_i$ values to deal with the exchange-symmetry, this is done with the following constraints
\begin{equation}
\sum_{1\leq t \leq \fT'} x^t_\ell \geq x^t_{\ell+1}\quad \forall a_\ell \in A\setminus \{a_n\} : w(a_\ell)=w(a_{\ell+1}), f_\ell=f_{\ell+1}, 1\leq t \leq \fT. \label{eq:sym2}
\end{equation}
They ensure that for two symbols $a_\ell$, $a_{\ell+1}$ with the same $w(a_i)$ and $f_i$ values, the one with smaller index must be placed first in the solution sequence.

Next, we present another set of valid inequalities, which are based on an extended model, i.e., a model with an additional set of variables.
Let binary variables $d_i^j$, $a_i \in A$, $1\leq j \leq M^*_i$ indicate the number of times symbol $a_i$ is appearing in the sequence $S$ encoded by the $x$-variables. The following constraints link the $x$-variables and the $d$-variables.

\begin{equation}
\sum_{1\leq t \leq \fT} x^t_i - \sum_{1\leq j \leq M^*_i} jd^j_i=0 \quad \forall a_i \in A \label{eq:d}
\end{equation}

With these $d$-variables, we can pose the following set of constraints \eqref{eq:d2}, which encode the connection between the number $k_i=|a_i(S)|$ of occurrences of a symbol $a_i$ in a sequence $S$ and the minimal possible weight-distance product induced by $a_i$, using the $D'_i=\lceil\fT/k_i \rceil$ relationship we discussed above in connection with the lifting of inequalities \eqref{eq:chosenmin}.

\begin{observation}
	Let $z^*$ be an upper bound on the objective value of an \WFS instance. Then the following inequalities are valid for our model extended with constraints \eqref{eq:d}.
	\begin{equation}
	P-z^*+\sum_{1 \leq j \leq M_i: w(a_i)\lceil\fT/j \rceil \leq z^*} (z^*-w(a_i)\lceil\fT/j \rceil)d^j_i
 \geq 0 \quad \forall a_i \in A \label{eq:d2}
	\end{equation}
	
\end{observation}
	
\section{Computational results \label{sec:results}}

The solution algorithm was implemented in C++ and IBM ILOG CPLEX 12.10 was used to solve the MIP models. The computations were made on a single core of an Intel Xeon E5-2670v2 machine with 2.5 GHz and 3GB of RAM, and all CPLEX settings were left on their default values. 

%We set a timelimit 1200 seconds for our all runs. 

\subsection{Instances}

\newcommand{\nI}{\texttt{normal}\xspace}
\newcommand{\lI}{\texttt{large}\xspace}

We used the instances introduced in \cite{pessoa2018weighted}, which are available at \url{https://sites.google.com/site/weightedfairsequencesproblem/instances}. There are two instance sets available at this site.
\begin{itemize}
	\item \nI: This was the main set used in the computational study of \cite{pessoa2018weighted}.
 The instance set consists of ten instances for each combination of $n=5,7,9,11,13,15$ (number of symbols) and $T=2n,3n,4n$, resulting in 180 instances in total. They were created by giving random integer weights from the interval $[1,2n]$ to each symbol, and the frequency $f_i=1$ for all symbols $a_i \in A$. 
 \item \lI: This instance set was created by the authors of \cite{pessoa2018weighted} to push the limits of their proposed approach. This instance set consists of ten instances for various combination of $n\in \{5,7,\ldots,15, 20,\ldots,50\}$ and $T\in \{2n, 3n,\ldots,10n,15n,\ldots, 40n\}$, resulting in 260 instances in total (see Table \ref{ta:large} for the used combinations). The weights and frequencies were created in the same fashion as in instance set \nI.
\end{itemize}

\subsection{Results}

We are first interested in the influence of the various ingredients of our solution framework. To do so, we compared the following three settings with a timelimit of 1800 seconds on instance set \nI.

\newcommand{\basic}{\texttt{basic}\xspace}
\newcommand{\ineqs}{\texttt{ineqs}\xspace}
\newcommand{\enhanced}{\texttt{enhanced}\xspace}

\begin{itemize}
	\item \basic: This is a basic setting consisting of Algorithm \ref{alg} used with our MIP model without any enhancements.
	\item \ineqs: In this setting, the MIP gets enhanced by the valid inequalities \eqref{eq:ub}, the variable fixings described and lifting described in in Observations \ref{obs:nearfixing}, \ref{obs:objfixing}, \ref{obs:liftingrhs} and also the symmetry breaking constraints \eqref{eq:sym}, \eqref{eq:sym2}.
	\item \enhanced: This setting is setting \ineqs enhanced with the extended MIP model with inequalities \eqref{eq:d}\eqref{eq:d2} and the results of Observations \ref{obs:skipping}, \ref{obs:ub2}, which can allow us to skip solving some MIPs and to lift the right-hand-side of inequalities \eqref{eq:ub}.
\end{itemize}

Figure \ref{fig:runtime} shows a plot of the runtime for solving instance set \nI. We see that the additional ingredients improve the performance of our solution approach, and setting \enhanced manages to solve nearly all instances to optimality within the given timelimit.

\begin{figure}[h!tb]
	
	\centering
	\includegraphics[width=0.8\textwidth]{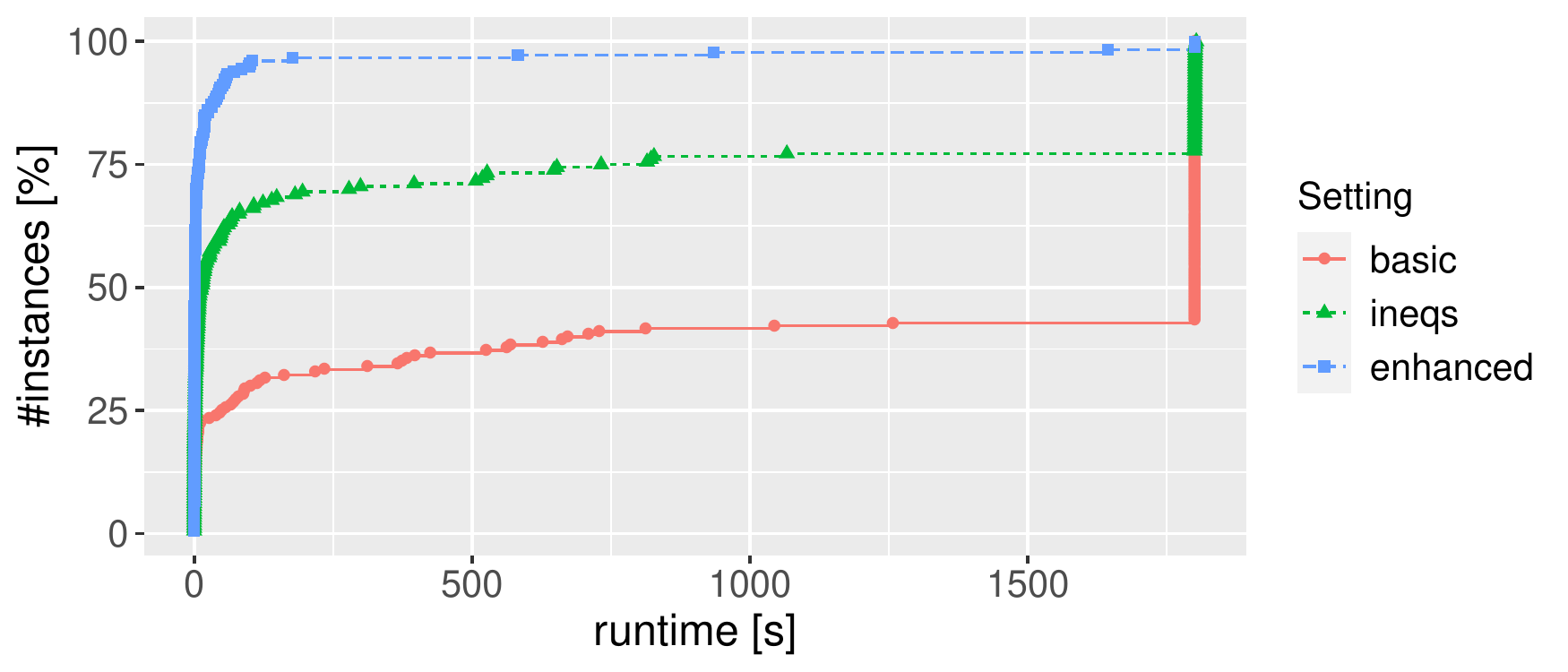}
	\caption{Runtime to optimality for instance set \nI\ and different settings.\label{fig:runtime}}
\end{figure}

Table \ref{ta:normal} gives the average runtime ($t [s]$) and number of solved instances ($\#nS$) aggregated by number of symbols and $T$ and also contains the results of the best approach of \cite{pessoa2018weighted}\footnote{the raw-data of their experiments is available at \url{https://sites.google.com/site/weightedfairsequencesproblem/computational-experiments}}. The runs in \cite{pessoa2018weighted} were made with CPLEX 12.6 on an Intel Core i7 with 1.9 GHz and 6GB of RAM, and as timelimit, they used 3600 seconds for all instances, except for the combinations $n=13, T=52$ and $n=15, T=60$ for which they used 18000 seconds.

\begin{table}[h!tb]
	\centering
	\caption{Results for instance set \nI aggregated by number of symbols and $T$ \label{ta:normal}.} 
	\begingroup\footnotesize
	\begin{tabular}{rr|rr|rr|rr|rr}
		\toprule
		& & \multicolumn{2}{|c}{\basic} & \multicolumn{2}{|c}{\ineqs}  & \multicolumn{2}{|c}{\enhanced} & \multicolumn{2}{|c}{\cite{pessoa2018weighted}}   \\ $n$ & $T$ & t [s] & $\#nS$  & t [s] & $\#nS$ & t [s] & $\#nS$  & t [s] &$\#nS$  \\ \midrule
		5 & 10 & 0.33 & 10 & 0.12 & 10 & 0.04 & 10 & 0.45 & 10 \\ 
		5 & 15 & 1.67 & 10 & 0.38 & 10 & 0.10 & 10 & 1.51 & 10 \\ 
		5 & 20 & 4.25 & 10 & 0.89 & 10 & 0.31 & 10 & 18.58 & 10 \\ 
		7 & 14 & 4.73 & 10 & 0.55 & 10 & 0.11 & 10 & 0.76 & 10 \\ 
		7 & 21 & 121.15 & 10 & 2.45 & 10 & 0.55 & 10 & 15.85 & 10 \\ 
		7 & 28 & 971.58 & 7 & 25.43 & 10 & 5.71 & 10 & 301.93 & 10 \\ 
		9 & 18 & 119.09 & 10 & 2.60 & 10 & 0.24 & 10 & 1.54 & 10 \\ 
		9 & 27 & 1617.55 & 2 & 7.60 & 10 & 0.91 & 10 & 28.33 & 10 \\ 
		9 & 36 & TL & 0 & 181.18 & 10 & 13.13 & 10 & 1475.92 & 8 \\ 
		11 & 22 & 740.50 & 8 & 13.31 & 10 & 0.53 & 10 & 2.46 & 10 \\ 
		11 & 33 & TL & 0 & 293.12 & 10 & 2.89 & 10 & 184.12 & 10 \\ 
		11 & 44 & TL & 0 & 940.85 & 5 & 189.40 & 9 & 2529.78 & 5 \\ 
		13 & 26 & TL & 0 & 48.78 & 10 & 1.33 & 10 & 6.82 & 10 \\ 
		13 & 39 & TL & 0 & 1052.43 & 5 & 10.26 & 10 & 1354.14 & 10 \\ 
		13 & 52 & TL & 0 & TL& 0 & 390.60 & 9 & 13013.13 & 4 \\ 
		15 & 30 & TL & 0 & 506.48 & 9 & 3.05 & 10 & 10.06 & 10 \\ 
		15 & 45 & TL & 0 & TL & 0 & 30.78 & 10 & 1471.48 & 7 \\ 
		15 & 60 & TL & 0 & TL & 0 & 377.60 & 8 & 15903.48 & 2 \\ 
		\bottomrule
	\end{tabular}
	\endgroup
\end{table}

The results show that our best setting \enhanced manages to solve 176 of 180 instances of this set. Moreover, for the combinations of $n$ and $T$ were the setting manages to solve all of the instances, the average runtime is at most 30 seconds. In contrast to this, the best setting of \cite{pessoa2018weighted} only manages to solve 154 instances to "optimality" (recall that the exactness of their algorithm relies on an unproven conjecture), with runtimes up to 1350 seconds for combinations of $n$ and $T$ were their approach manages to solve all the instances to "optimality". 

The results for instance set \nI also show that our approach begins to struggle when $T=4n$ for $n=11,13,15$. Thus, next we are interested in the results for instance set \lI which contains challenging instances, where some have larger $n$ and $T$ is created based on a larger multiplication factor for $n$. Table \ref{ta:large} gives the result for our approach with setting \enhanced and the best approach from \cite{pessoa2018weighted}, the timelimit was set to 3600 seconds for these runs.
\begin{table}[ht]
	\centering
	\caption{Results for instance set \lI aggregated by number of symbols and $T$ \label{ta:large}.} 
	\begingroup\footnotesize
	\begin{tabular}{rr|rr|rr}
		\toprule
		&  & \multicolumn{2}{|c}{\enhanced} & \multicolumn{2}{|c}{\cite{pessoa2018weighted}}   \\ $n$ & $T$ & t [s] & $\#nS$  & t [s] & $\#nS$  \\ \midrule
		5 & 25 & 2.35 & 10 & 384.44 & 10 \\ 
		5 & 30 & 3.05 & 10 & 1305.70 & 8 \\ 
		5 & 35 & 10.98 & 10 & 2281.99 & 5 \\ 
		5 & 40 & 173.92 & 10 & 2658.08 & 3 \\ 
		5 & 50 & 407.71 & 10 & 1823.73 & 5 \\ 
		5 & 75 & 837.11 & 8 & 2178.25 & 4 \\ 
		5 & 100 & 450.99 & 9 & 2648.17 & 3 \\ 
		5 & 125 & 1518.29 & 6 & 2620.63 & 3 \\ 
		5 & 150 & 1618.53 & 6 & 2962.30 & 3 \\ 
		5 & 200 & 2186.13 & 4 & TL & 0 \\ 
		7 & 35 & 79.36 & 10 & 2334.06 & 5 \\ 
		7 & 42 & 91.21 & 10 & 2995.13 & 2 \\ 
		7 & 49 & 566.62 & 9 & 2503.26 & 4 \\ 
		7 & 56 & 798.94 & 8 & 3142.86 & 2 \\ 
		7 & 63 & 1595.93 & 6 & TL & 0 \\ 
		9 & 45 & 802.45 & 8 & TL & 0 \\ 
		11 & 55 & 674.43 & 9 & TL & 0 \\ 
		20 & 40 & 11.57 & 10 & 28.07 & 10 \\ 
		20 & 60 & 116.78 & 10 & 2933.50 & 3 \\ 
		25 & 50 & 49.27 & 10 & 115.55 & 10 \\ 
		25 & 75 & 763.45 & 10 & TL& 0 \\ 
		30 & 60 & 140.12 & 10 & 498.84 & 9 \\ 
		35 & 70 & 383.09 & 10 & 1310.33 & 8 \\ 
		40 & 80 & 763.53 & 10 & 2693.67 & 5 \\ 
		45 & 90 & 1976.25 & 10 & 3445.86 & 3 \\ 
		50 & 100 & 3180.80 & 5 & TL & 0 \\ 
		\bottomrule
	\end{tabular}
	\endgroup
\end{table}

We see that our approach manages to solve 228 of 260 instances to optimality, while the approach of \cite{pessoa2018weighted} only solves 105 instances. Moreover, our approach manages to solve at least one instance for each combination of $n$ and $T$, while there are six such combinations where the approach of \cite{pessoa2018weighted} cannot solve any instance. The problem seems to become harder when $n$ and $T$ gets larger, this is not unexpected, as in this case also the model gets larger and more iterations have to be computed in the algorithm. Detailed results for both instance sets are made available in machine-readable format at \url{https://msinnl.github.io/pages/instancescodes.html}.

\section{Conclusion and outlook \label{sec:conclusion}}

In this paper, we presented a new iterative exact solution algorithm for the weighted fair sequences problem (\WFS), which is a recently introduced NP-hard sequencing problem with applications in diverse areas such as TV advertisement scheduling, periodic machine maintenance and real-time scheduling. In the \WFS we are given an upper bound on the allowed solution sequence length and a list of symbols. For each symbols, there is a positive weight and a number, which gives the minimum times the symbol has to occur in a feasible solution sequence. The goal is to find a feasible sequence, which minimizes the maximum weight-distance product, which is calculated for each consecutive appearance of each symbol in the sequence, including the last and first appearance in the sequence. Our solution algorithm is based on a new mixed-integer programming model for a fixed sequence length. This model is then solved iteratively. We develop valid inequalities, variable fixings and an extended model which allows to pose additional valid inequalities. We also present a result which can allow to skip solving some iterations. In a computational study on instances from literature, we see that our approach vastly outperforms existing solution approaches from literature.

There are various avenues for further work. The development of additional valid inequalities could be interesting to further speed up the solution process. The design of effective heuristics could also be a worthwhile topic, as feasible solutions can be used to initialize our solution algorithm. Studying alternative modeling approaches, such as a column generation/branch-and-price based solution approach were there is a set of variables for each label, which encodes the positions in the sequence this label is placed could also be fruitful. 

\section*{Acknowledgments}

The research was supported by the Linz Institute of Technology (Project 
LIT-2019-7-YOU-211) and the JKU Business School.

\section*{References}

\bibliographystyle{elsarticle-harv}
\bibliography{wfs-bib}

%\appendix{Detailed results}

\end{document}